\documentclass{article}%
\usepackage{amsfonts}
\usepackage{amsmath}
\usepackage{geometry}
\usepackage{amssymb}
\usepackage{graphicx}%
\providecommand{\U}[1]{\protect\rule{.1in}{.1in}}
\newtheorem{theorem}{Theorem}

\newtheorem{corollary}[theorem]{Corollary}

\newtheorem{definition}[theorem]{Definition}

\newtheorem{remark}[theorem]{Remark}

\newenvironment{proof}[1][Proof]{\noindent\textbf{#1.} }{\ \rule{0.5em}{0.5em}}
\begin{document}

\title{Analysis of the Fractional Integrodifferentiability of Power Functions and
Hypergeometric Representation}
\author{Rodrigues F.G$^{1}$. and Capelas de Oliveira E$^{2}$.\\1- Departamento de Matem\'{a}ticas, Universidad de La Serena, Chile.\\2- Departamento de Matem\'{a}tica Aplicada, IMECC-UNICAMP, Brasil.}
\maketitle

\begin{abstract}
Fractional calculus has gained enormous visibility among the scientific
community in last few decades and yet there are still many controversial and
open theoretical discussions. In this work we want to elucidade in a simple
way that it is possible to calculate the fractional integrals and derivatives
of order $\alpha$ (using the Riemann-Liouville formulation) of power functions
$\left(  t-\ast\right)  ^{\beta}$ with $\beta$ being any real value, so long
as one pays attention to the proper choosing of the lower and upper limits
according to the original function's domain. This is specially oriented as a
review to those newly inducted on the process fractionally
integrodifferentiating under the (in view of the nowadays multiples
definitions) \textquotedblleft classical\textquotedblright\ Riemann-Liouville
setting. We, therefore, obtain valid expressions that are described in terms
of function series of the type $\left(  t-\ast\right)  ^{\pm\alpha+k}$ and
also observe that they are related to the famous hypergeometric functions of
the Mathematical-Physics.

\end{abstract}

\section{Introduction}

The non-integer order calculus, popularly known as fractional calculus (FC)
was born in 1695. Only after nearly 250 years did the first event dedicated
exclusively to the theme \cite{bross1,Ross1974,bross2}. Today, after more than
40 years, FC has gained enormous visibility both from a theoretical point of
view and in applications.

In the theoretical point of view we mention, in addition to the classical
formulations of the derivative (Riemann-Liouville, Caputo and
Gr\"unwald-Letnikov, just to cite a few) \cite{ecotenreiro} a recent one due
to Caputo-Fabrizio \cite{caputofabrizio} where the kernel is non singular and
whose properties were studied by Losada-Nieto \cite{losadanieto} and as an
application by Atangana \cite{atangana} in the study of Fischer's
reaction-diffusion equation. As a generalization of these derivatives we cite
the $\Psi$-Hilfer derivative \cite{ZeEco} and the general Hilfer-Hadamard
derivative \cite{DanielaEco}. On the other hand, also recent, there is a wide
class of new derivatives that recover the classical (definition) derivative in
the Newton/Leibniz sense and, even if fractional call, they are local
derivatives \cite{ZeEcolocal}.

As applications we can cite a wide range where the FC acts. Some of them are:
Yang et al. \cite{yangtenreirobzleanu} discuss anomalous diffusion models with
general fractional derivatives within the kernels of the extended
Mittag-Leffler functions; Yang \cite{yang} propose a fractional derivative of
constant and variable orders applied to anomalous relaxation models in
heat-transfer problems; Atangana and Baleanu \cite{atanganabaleanu} introduce
a new fractional derivative with non-local and non-singular kernel to discuss
a particular heat transfer model; Yang et al. \cite{yangtenreironieto} present
a new family of the local fractional partial differential equations; Yang et
al. \cite{yangtenreirocattanigao} discuss a $LC$-electric circuit modeled by a
local fractional calculus; Gao and Yang \cite{gaoyang} using local fractional
Euler's method, discuss the steady heat-conduction problem; G\'{o}mez and
Capelas de Oliveira \cite{Gomez} propose and discuss a nonlinear partial
differential equation variational iteraction method; Costa et al.
\cite{felixetal} present a nonlinear fractional Harry Dym equation whose
solution is written in terms of the Fox's $H$-function; Garrappa et al.
\cite{garrappa} discus models of dielectric relaxation based on completely
monotone functions; Rosa and Capelas de Oliveira \cite{rosaEco}, discuss some
particular fractional differential equation involving dielectrics and discuss
the complete monotonicity and Capelas de Oliveira et al. \cite{capelasmenon}
discuss analytic components for the hadronic total cross-section using Mellin transform.

When it comes to \textquotedblleft Calculus and Analysis\textquotedblright%
\ concepts such as limits, infinitesimals and continuity of functions are
among the first that, at one step after another, leads to the definitions of
the operations of differentiation and integration. Since these operations acts
on a certain suitable class of functions, it is due to their simplicity and
well understood behavior that in any introductory course to
(Standard)\ Calculus the first class of functions to be \textquotedblleft
concretely\textquotedblright\ computed are the power functions, that is,
functions of the type $f(t)=\left(  t-d\right)  ^{\beta}$, with $d,\beta\in%
\mathbb{R}
$. It is therefore logical that in the theory of the so-called
\emph{Fractional Calculus} (FC) the rules for operating these type of
functions must also be carefully analyzed and established, but there is more
to it than this simple argument. Power functions defines a very important
class of functions and there are many natural (and artificial) phenomena that
behaves according to power law distributions (e.g., Newtonian laws of gravity
and electrostatics, Kepler's third law about the orbital period of planets,
the square-cube law relating the ratio between surface area and volume of an
object, the Stefan-Boltzman law describing the power radiated from a black
body in terms of its temperature, etc...), particularly, there are an
increasing number of works showing the interesting link between FC models and
physical phenomena described in terms of power laws. Just to mention a few
examples, in \cite{Chen} it is presented a new fractional Laplacian time-space
model to describe the frequency-dependent attenuation obeying empirical power
law distribution. In \cite{Korabel} it is studied the duality of Hamiltonian
dynamics of a system of particles with power-like interactions with the
solution of certain fractional differential equations. In \cite{Caputo1} is
presented a fractional generalization of the Kelvin-Voigt rheology in order to
better simulate the power law stress-strain relation of some biological media.
In \cite{Nasholm} the authors study the link between multiple relaxation
models, power law attenuation and fractional wave equations, providing some
physically based evidence for the use of FC in the modelling processes. In
\cite{Ionescu} it is proposed a fractional model to describe a power law
relation between pain transmission processes in the human body and analgesia
measurements, and the lists may go on. So all these works suggest that having
a clear picture of the theoretical behavior and \textit{modus operandis}\ of
FC operators acting on power functions (and more generally to functions
described in terms of power series such as analytic functions) may provide us
with insights to several distinct power law related phenomena or, at the very
least, a better and concise mathematical tool available to be used in the same
way as already happens with the standard calculus.

In fact, in this work we are concerned exactly to this matter, because
although the main literature provide some basic rules for fractionally
operating power functions, we believe that there are still some fundamental
aspects that are not fully cleared and understood and the matter still require
more attention. To clarify this argument, we recall the following FC result:
Let $f(t)=\left(  t-d\right)  ^{\beta}$, with $d,\beta\in%
\mathbb{R}
$. It is known that whenever $\beta>-1$, we can compute expressions for the
Riemann-Liouville fractional integral (RLFI) and Riemann-Liouville fractional
derivative (RLFD) of these power functions \cite{Kilbas}:%
\begin{align}
\left[  _{d}J_{t}^{\alpha}\left(  x-d\right)  ^{\beta}\right]  (t)  &
=\frac{\Gamma\left(  \beta+1\right)  }{\Gamma\left(  \beta+\alpha+1\right)
}\left(  t-d\right)  ^{\beta+\alpha},\label{Ipower}\\
\left[  _{d}D_{t}^{\alpha}\left(  x-d\right)  ^{\beta}\right]  (t)  &
=\frac{\Gamma\left(  \beta+1\right)  }{\Gamma\left(  \beta-\alpha+1\right)
}\left(  t-d\right)  ^{\beta-\alpha}, \label{Dpower}%
\end{align}
for $\operatorname{Re}\left(  \alpha\right)  \geq0$.

The proofs for the expressions in Eq.(\ref{Ipower}) and Eq.(\ref{Dpower}) are
very straightforward and can be found, e.g. in \cite{Kai} and they make use of
the integral representation for the Euler's beta function%
\begin{align}
B(\eta,\xi)  &  =\int_{0}^{1}x^{\eta-1}\left(  1-x\right)  ^{\xi
-1}dx,\label{BetaA}\\
&  =\frac{\Gamma\left(  \eta\right)  \Gamma\left(  \xi\right)  }{\Gamma\left(
\eta+\xi\right)  },\text{ }\operatorname{Re}[\eta],\operatorname{Re}[\xi]>0.
\label{BetaB}%
\end{align}

But here, we are particularly interested in evaluating possible expressions
for the RLFI and the RLFD of order $\alpha\in%
\mathbb{R}
$ for the power functions stated above without restricting the value of the
index $\beta\in%
\mathbb{R}
$, just as it happens when we're dealing with the integer order calculus.
Obviously, as one might expect, we won't be able to use the same direct kind
of Euler's beta function approach, because of the restriction of the integral
representation of the Euler's beta function to the positive half
complex-plane. We also point out that while the standard literature provide
expressions for the RLFI and RLFD with the lower limit $a$ of the operators
coinciding with the shifting factor $d$ in the argument of the power function
(i.e., when $a=d$ we call this situation, \textquotedblleft
centered\textquotedblright), we are also concerned with a more general setting
when $a$ and $d$ are not necessarily equal (in such cases, we call this
situation \textquotedblleft displaced\textquotedblright). This is particularly
important, because such \textquotedblleft displaced\textquotedblright%
\ expressions frequently occurs when trying to solve \emph{fractional delay
differential equations} \cite{Morgado}.

So in this work, we show that when one takes careful consideration on the
choices for the lower and upper limits of these operations, it is possible to
compute expressions for the RLFI and RLFD of any power function (regardless of
the index of the power) in terms of series that can be related to the famous
hypergeometric functions \cite{Lavoie} so important and commonly found in many
problems of the Mathematical-Physics. It's worth mentioning that in
\cite{Andriam} the authors have provided two alternative definitions for
fractional derivatives of power functions of any order, but they approached
the problem in a very distinct way as they have not used the Riemann-Liouville
formulation. Also, while there are many distinct formulations for a fractional
differential operator \cite{ecotenreiro, caputofabrizio, Khalil, Ortigueira},
it is our hope that with this work we can, not only provide helpful
expressions for the aforementioned calculations of RLFI and RLFD of power
functions to be used on analytical or numerical related problems, but also set
some ground for a future discussion on the theoretical aspects of computing a
\emph{fractional definite integral }versus knowing its \emph{fractional
primitives} whenever such definitions are meaningful.

We recall some basic concepts in the preliminary section, in the second and
third ones we present the main results stated as theorems, we then provide a
summary of the results for convenience compiling the main formulas obtained.
Finally we draw our conclusions and expose some further topics for researches.
This work also contain an appendix where it is shown some calculations where
we point out that our expressions can be identified with some hypergeometric functions.

\section{Preliminaries}

We recall that the RLFD is defined in terms of the RLFI which is, by the way,
defined in terms of a \emph{\textquotedblleft ordinary\textquotedblright
definite integral}, (see Def.\ref{Def1} and Def.\ref{Def2}). So by
construction, not only this makes the operator non-local, but also address the
matter of fractional integrodifferentiability of a function to be dependant on
the (upper and lower) limits of integration and the domain in which we want to
operate, just as it happens in the classical integer order theory and we know
that we can calculate the (classic integer order) \emph{indefinite integral}
of power functions $f(t)=\left(  t-d\right)  ^{\beta}$ regardless the values
of $\beta\in%
\mathbb{R}
$. This means that any power function has a \textquotedblleft
classical\textquotedblright\ primitive and, as a consequence, one can
calculate the $n$-fold (indefinite) integral (or $n$-fold primitive) of any
order $n\in%
\mathbb{N}
$. The situation is similar when calculating the $nth$-order derivatives.
Since the domain of the resulting power functions after the operations does
eventually change and depend on the index $\beta$ as well as the order $n$ of
the operators, then being (or not) able to calculate the \emph{definite
integrals} of these functions actually depends on \emph{where} we perform the
operations. After all in order to concretely calculate any \textquotedblleft
definite\textquotedblright\ results we must use the expressions for the
functions obtained by the action of the formal operators: $n$-fold
\emph{indefinite integrals} and the $nth$-order derivatives. So, even though
we can perform these integer order operators on power functions of any order,
it is not necessarily true that one can calculate any \emph{definite
integrals} or \emph{derivative at a point} for these functions at any
arbitrary interval of $%
\mathbb{R}
$. And we point out that the situation for the fractional case should be similar.

\begin{definition}
\label{Def1}Let $\Omega=\left[  a,b\right]  \subset%
\mathbb{R}
$, $f\in L_{1}\left(  \Omega\right)  $ and $\alpha\in%
\mathbb{R}
_{+}$. The expression for $\left[  _{a}J_{t}^{\alpha}f\left(  x\right)
\right]  \left(  t\right)  $,%
\begin{equation}
\left[  _{a}J_{t}^{\alpha}f\left(  x\right)  \right]  \left(  t\right)
=\frac{1}{\Gamma\left(  \alpha\right)  }\int_{a}^{t}\left(  t-x\right)
^{\alpha-1}f(x)dx, \label{RLFI}%
\end{equation}
defines \emph{(}leftwise\emph{)} the Riemann-Liouville fractional integral of
order $\alpha$ of the function $f$.
\end{definition}

\begin{definition}
\label{Def2}Let $\Omega=\left[  a,b\right]  \subset%
\mathbb{R}
$, $\mathbf{n}=\left[  \alpha\right]  +1$ where $\left[  \alpha\right]  $ is
the integer part of $\alpha\in%
\mathbb{R}
_{+}$ and $f\in AC^{\mathbf{n}}\left(  \Omega\right)  $. The expression for
$\left[  _{a}D_{t}^{\alpha}f\left(  x\right)  \right]  \left(  t\right)  $,%
\begin{align}
\left[  _{a}D_{t}^{\alpha}f(x)\right]  (t)  &  =\frac{d^{\mathbf{n}}%
}{dt^{\mathbf{n}}}\left[  _{a}J_{t}^{\mathbf{n}-\alpha}f(x)\right]
(t)\nonumber\\
&  =\frac{d^{\mathbf{n}}}{dt^{\mathbf{n}}}\frac{1}{\Gamma\left(
\mathbf{n}-\alpha\right)  }\int_{a}^{t}\left(  t-x\right)  ^{\mathbf{n}%
-\alpha-1}f(x)dx \label{RLFD}%
\end{align}
defines \emph{(}leftwise\emph{)} the Riemann-Liouville fractional derivative
of order $\alpha$ of the function $f$.
\end{definition}

One can similarly define the (rightwise) versions of the RLFI and RLFD
\cite{Podlubny, Rodrigues}.

In the definitions for the RLFI (Eq.(\ref{RLFI})) and the RLFD (Eq.(\ref{RLFD}%
)), if $\alpha=0$, we define both operators to be the identity operator
$\mathbf{I}$, while if one chooses $\alpha=n\in%
\mathbb{N}
$ both operators reduces, respectively, to their integer order counterparts,
that is, as $\alpha\rightarrow n$, $\left[  _{a}J_{t}^{\alpha}f\left(
x\right)  \right]  \left(  t\right)  $ $\rightarrow\left[  _{a}J_{t}%
^{n}f\left(  x\right)  \right]  \left(  t\right)  $ the $n$-fold integral and
$\left[  _{a}D_{t}^{\alpha}f(x)\right]  (t)\rightarrow\left[  _{a}D_{t}%
^{n}f(x)\right]  (t)=\frac{d^{n}}{dt^{n}}f(t)$ the $nth$-order derivative
\cite{Kai, Oldham}.

In both definitions, the $\Gamma\left(  \ast\right)  $ symbol refers to the
gamma function and we will be using some of its properties related to the
ascending and descending Pochhammer symbols defined, respectively, as:%
\begin{equation}
\left(  z\right)  _{k}=\left\{
\begin{array}
[c]{lc}%
1, & \text{if }k=0,\\
z(z+1)\cdots(z+k-1), & \text{if }k\in%
\mathbb{N}
.
\end{array}
\right.  \label{PocA}%
\end{equation}%
\begin{equation}
\left(  z\right)  _{-k}=\left\{
\begin{array}
[c]{lc}%
1, & \text{if }k=0,\\
z(z-1)\cdots(z-k+1), & \text{if }k\in%
\mathbb{N}
,
\end{array}
\right.  \label{PocD}%
\end{equation}
which can be rewritten in the following form%
\begin{align}
\left(  z\right)  _{k}  &  =\frac{\Gamma\left(  z+k\right)  }{\Gamma\left(
z\right)  },\label{IdA}\\
\left(  z\right)  _{-k}  &  =\frac{\Gamma\left(  z+1\right)  }{\Gamma\left(
z-k+1\right)  }, \label{IdB}%
\end{align}
and are valid the relations below%

\begin{align}
\left(  -z\right)  _{k}  &  =\left(  -1\right)  ^{k}\left(  z\right)
_{-k},\label{IdAA}\\
\left(  -z\right)  _{-k}  &  =\left(  -1\right)  ^{k}\left(  z\right)  _{k}.
\label{IdBB}%
\end{align}

We also recall that the gamma function is uniquely determined as the function
satisfying the functional relation%
\begin{align}
\Gamma\left(  1\right)   &  =1,\label{Func}\\
z\Gamma\left(  z\right)   &  =\Gamma\left(  z+1\right)  ,\text{ }%
\operatorname{Re}\left(  z\right)  >0, \label{Func1}%
\end{align}
but the relation on Eq.(\ref{Func1}) can be used to extend it analytically to
all complex values, except on $%
\mathbb{N}
_{0}=\left\{  0,-1,-2,\ldots\right\}  $, where $\Gamma\left(  -n\right)
\rightarrow\pm\infty$, $n\in%
\mathbb{N}
$. Yet, the relation on Eq.(\ref{Func1}) is valid for all complex values and
when dealing with the elements of $%
\mathbb{N}
_{0}$ one should consider%
\begin{equation}
z^{\pm}\Gamma\left(  z^{\pm}\right)  =\Gamma\left(  z^{\pm}+1\right)  .
\label{Func2}%
\end{equation}

Although the gamma function is not defined for negative integers, the ratio of
gamma functions of negative integers are defined \cite{Oldham}%
\begin{equation}
\frac{\Gamma\left(  -n\right)  }{\Gamma\left(  -m\right)  }=\left(  -1\right)
^{m-n}\frac{m!}{n!},\text{ }m,n\in%
\mathbb{N}
, \label{IdC}%
\end{equation}
\newline and we point out that Eq.(\ref{IdC}) is also valid when choosing $m$
or $n$ to be zero, with%
\begin{align}
\frac{\Gamma\left(  0\right)  }{\Gamma\left(  -m\right)  }  &  =\left(
-1\right)  ^{m}\frac{m!}{0!}=\left(  -1\right)  ^{m}m!,\label{A}\\
\frac{\Gamma\left(  -n\right)  }{\Gamma\left(  0\right)  }  &  =\left(
-1\right)  ^{-n}\frac{0!}{n!}=\frac{\left(  -1\right)  ^{n}}{n!},\label{B}\\
\frac{\Gamma\left(  0\right)  }{\Gamma\left(  0\right)  }  &  =1. \label{C}%
\end{align}

Finally, due to the definition and properties of the (analytically extended)
gamma function and its relation with the Pochhammer symbols above, we are
allowed to generalize the binomial coefficients to non integer values%
\begin{equation}
\left(
\begin{array}
[c]{c}%
\beta\\
k
\end{array}
\right)  =\frac{\Gamma\left(  \beta+1\right)  }{\Gamma\left(  \beta
-k+1\right)  \Gamma\left(  k+1\right)  },\text{ }\beta\in%
\mathbb{R}
\text{ and }k\in%
\mathbb{N}
_{0}\text{,} \label{Bin}%
\end{equation}
where we denote $%
\mathbb{N}
_{0}=%
\mathbb{N}
\cup\left\{  0\right\}  $.

In this work, we will restrict ourselves to $0<\alpha<1$. This will simplify
the main analysis without loosing great generality, since for any arbitrary
order $\alpha\in%
\mathbb{R}
$, we have $\alpha=\left[  \alpha\right]  +\left\{  \alpha\right\}  $, where
$\left[  \alpha\right]  $ is the integer part of $\alpha$ and $0<\left\{
\alpha\right\}  <1$ its fractional part and the main feature of FC is related
exactly to the non-integer part. Hence we are considering the evaluation of
these expressions:%

\begin{align}
\left[  _{a}J_{t}^{\alpha}\left(  x-d\right)  ^{\beta}\right]  (t)  &
=\frac{1}{\Gamma\left(  \alpha\right)  }\int_{a}^{t}\left(  t-x\right)
^{\alpha-1}(x-d)^{\beta}dx\label{1}\\
\left[  _{a}D_{t}^{\alpha}\left(  x-d\right)  ^{\beta}\right]  (t)  &
=\frac{d}{dt}\left[  _{a}J_{t}^{1-\alpha}\left(  x-d\right)  ^{\beta}\right]
(t)\nonumber\\
&  =\frac{d}{dt}\frac{1}{\Gamma\left(  1-\alpha\right)  }\int_{a}^{t}\left(
t-x\right)  ^{-\alpha}(x-d)^{\beta}dx. \label{2}%
\end{align}

But before we begin with the explicit calculations, lets call attention to the
domain in $%
\mathbb{R}
$ of $f(t)=\left(  t-d\right)  ^{\beta}$, which essentially depend on the
values of $d$ and $\beta$ as it plays an important role on (definite)
integrability and differentiability (at a point).

We have the following possibilities:

\textbf{Case 1: }%
\begin{equation}
Dom(f)=\left\{
\begin{array}
[c]{l}%
\mathbb{R}
,\text{ if }\beta\in%
\mathbb{N}
_{0},\\%
\mathbb{R}
\setminus\{d\},\text{ if }\beta\in%
\mathbb{Z}
\setminus%
\mathbb{N}
_{0}.
\end{array}
\right.  \label{3}%
\end{equation}

\textbf{Case 2: }If $\beta\in%
\mathbb{Q}
\setminus%
\mathbb{Z}
$, that is, if $\beta$ is a proper rational fraction, then we can assume
without loss of generality that $\beta=\frac{p}{q}$ with $p\in%
\mathbb{Z}
$ and $q\in%
\mathbb{N}
$ and under this hypothesis, we know that%
\begin{equation}
f(t)=\left(  t-d\right)  ^{\beta}=\sqrt[q]{\left(  t-d\right)  ^{p}},
\label{4}%
\end{equation}
and since the operation $\sqrt[q]{\left(  \ast\right)  }$, $q\in%
\mathbb{N}
$ is well defined in $%
\mathbb{R}
_{+}^{0}=\left[  0,+\infty\right)  $, we conclude that:%
\begin{equation}
Dom(f)=\left\{
\begin{array}
[c]{l}%
\mathbb{R}
,\text{ if }p\ \text{is an even positive integer;}\\
\left[  d,+\infty\right)  ,\text{ if }p\text{ is an odd positive integer;}\\
\left(  d,+\infty\right)  ,\text{ if }p\text{ is a negative integer (even or
odd).}%
\end{array}
\right.  \label{5}%
\end{equation}

\textbf{Case 3:} Now if $\beta\in%
\mathbb{R}
\setminus%
\mathbb{Q}
$, that is, if $\beta$ is irrational, then we can use the following identity:%
\begin{equation}
\left(  t-d\right)  ^{\beta}=e^{\ln\left(  t-d\right)  ^{\beta}}=e^{\beta
\ln\left(  t-d\right)  }, \label{6}%
\end{equation}
and since the domain in $%
\mathbb{R}
$ of the logarithm is $%
\mathbb{R}
_{+}=\left(  0,+\infty\right)  $ while the domain in $%
\mathbb{R}
$ of the exponential is $%
\mathbb{R}
$ itself, then we conclude that in such case $Dom(f)=\left(  d,+\infty\right)
$.

\section{Riemann-Liouville Integration of Power Functions}

We start with the following theorem.

\begin{theorem}
\label{Th1}Let $f(t)=\left(  t-d\right)  ^{\beta}$, $d,\beta\in%
\mathbb{R}
$ and suppose $t\in\Omega\subset%
\mathbb{R}
$, where $\Omega$ is an interval where $f$ is properly defined as real valued
function. Then\newline\emph{(I)}
\begin{equation}
\left[  _{a}J_{t}^{\alpha}\left(  x-d\right)  ^{\beta}\right]  (t)=\sum
_{k=0}^{\infty}\frac{\Gamma\left(  \beta+1\right)  \left(  a-d\right)
^{\beta-k}\left(  t-a\right)  ^{\alpha+k}}{\Gamma\left(  \beta-k+1\right)
\Gamma\left(  \alpha+k+1\right)  },\text{ }t\in\left[  a,a+\frac{\epsilon}%
{2}\right)  , \label{Ia}%
\end{equation}
\newline\emph{(II)}
\begin{equation}
\left[  _{d^{+}}J_{t}^{\alpha}\left(  x-d\right)  ^{\beta}\right]
(t)=\sum_{k=0}^{\mathbf{\infty}}\frac{\Gamma\left(  \beta+1\right)
\epsilon^{\beta-k}\left(  t-d^{+}\right)  ^{\alpha+k}}{\Gamma\left(
\beta-k+1\right)  \Gamma\left(  \alpha+k+1\right)  },\text{ }t\in\left[
d^{+},d^{+}+\epsilon\right)  . \label{Ib}%
\end{equation}
\newline Particularly, if $\beta=m\in%
\mathbb{N}
_{0}$, then%
\begin{equation}
\left[  _{a}J_{t}^{\alpha}\left(  x-d\right)  ^{m}\right]  (t)=%
{\displaystyle\sum\limits_{k=0}^{m}}
\frac{\Gamma\left(  m+1\right)  \left(  a-d\right)  ^{m-k}\left(  t-a\right)
^{\alpha+k}}{\Gamma\left(  m-k+1\right)  \Gamma\left(  \alpha+k+1\right)
},\text{ }a,t\in%
\mathbb{R}
, \label{Ic}%
\end{equation}
while if $\beta=-m$ \ with $m\in%
\mathbb{N}
$, then we can use alternatively the following expressions as well
\begin{align}
\left[  _{a}J_{t}^{\alpha}\left(  x-d\right)  ^{-m}\right]  (t)  &
=\sum_{k=0}^{\infty}\frac{\left(  -1\right)  ^{k}\Gamma\left(  m+k\right)
\left(  a-d\right)  ^{-\left(  m+k\right)  }\left(  t-a\right)  ^{\alpha+k}%
}{\Gamma\left(  m\right)  \Gamma\left(  \alpha+k+1\right)  },\label{Id}\\
\left[  _{d^{+}}J_{t}^{\alpha}\left(  x-d\right)  ^{-m}\right]  (t)  &
=\sum_{k=0}^{\mathbf{\infty}}\frac{\left(  -1\right)  ^{k}\Gamma\left(
m+k\right)  \epsilon^{-\left(  m+k\right)  }\left(  t-d^{+}\right)
^{\alpha+k}}{\Gamma\left(  m\right)  \Gamma\left(  \alpha+k+1\right)  },
\label{Ie}%
\end{align}
with $t\in\left[  a,a+\frac{\epsilon}{2}\right)  $ and $t\in\left[
d^{+},d^{+}+\epsilon\right)  $, respectively and where in all cases,
$\epsilon=\left\vert d-a\right\vert >0.$
\end{theorem}

\begin{proof}
Initially, let $\beta\in%
\mathbb{R}
\setminus%
\mathbb{Z}
$. Using the definition of the RLFI (Eq.(\ref{RLFI})) and integrating by parts
a total of $\mathbf{p}$ times, we get%
\begin{equation}
\left[  _{a}J_{t}^{\alpha}\left(  x-d\right)  ^{\beta}\right]  (t)=\sum
_{k=0}^{\mathbf{p}-1}\frac{\left(  \beta\right)  _{-k}\left(  a-d\right)
^{\beta-k}\left(  t-a\right)  ^{\alpha+k}}{\Gamma\left(  \alpha+k+1\right)
}+\mathcal{R}_{\mathbf{p}}, \label{IP1}%
\end{equation}
where%
\begin{equation}
\mathcal{R}_{\mathbf{p}}=\frac{\left(  \beta\right)  _{-\mathbf{p}}}%
{\Gamma\left(  \alpha+\mathbf{p}\right)  }\int_{a}^{t}\left(  x-d\right)
^{\beta-\mathbf{p}}\left(  t-x\right)  ^{\alpha+\mathbf{p}-1}dx. \label{IP2}%
\end{equation}
\newline Now, using the identity in Eq.(\ref{IdB}), we can write%
\begin{equation}
\left[  _{a}J_{t}^{\alpha}\left(  x-d\right)  ^{\beta}\right]  (t)=\sum
_{k=0}^{\mathbf{p}-1}\frac{\Gamma\left(  \beta+1\right)  \left(  a-d\right)
^{\beta-k}\left(  t-a\right)  ^{\alpha+k}}{\Gamma\left(  \beta-k+1\right)
\Gamma\left(  \alpha+k+1\right)  }+\mathcal{R}_{\mathbf{p}}, \label{IP1a}%
\end{equation}
with
\begin{equation}
\mathcal{R}_{\mathbf{p}}=\frac{\Gamma\left(  \beta+1\right)  }{\Gamma\left(
\beta-\mathbf{p}+1\right)  \Gamma\left(  \alpha+\mathbf{p}\right)  }\int%
_{a}^{t}\left(  x-d\right)  ^{\beta-\mathbf{p}}\left(  t-x\right)
^{\alpha+\mathbf{p}-1}dx. \label{IP2a}%
\end{equation}
\newline We now estimate the remainder $\mathcal{R}_{\mathbf{p}}$%
\begin{align}
\left\vert \mathcal{R}_{\mathbf{p}}\right\vert  &  \leq\left\vert \frac
{\Gamma\left(  \beta+1\right)  }{\Gamma\left(  \beta-\mathbf{p}+1\right)
\Gamma\left(  \alpha+\mathbf{p}\right)  }\right\vert \int_{a}^{t}\left\vert
x-d\right\vert ^{\beta-\mathbf{p}}\left\vert t-x\right\vert ^{\alpha
+\mathbf{p}-1}dx\nonumber\\
&  \leq\left\vert \frac{\Gamma\left(  \beta+1\right)  }{\Gamma\left(
\beta-\mathbf{p}+1\right)  \Gamma\left(  \alpha+\mathbf{p}\right)
}\right\vert \left\vert t-a\right\vert ^{\alpha+\mathbf{p}-1}\int_{a}%
^{t}\left\vert x-d\right\vert ^{\beta-\mathbf{p}}dx\nonumber\\
&  =\left\vert \frac{\Gamma\left(  \beta+1\right)  }{\Gamma\left(
\beta-\mathbf{p}+1\right)  \Gamma\left(  \alpha+\mathbf{p}\right)
}\right\vert \left\vert t-a\right\vert ^{\alpha+\mathbf{p}-1}\left\{
\frac{\left\vert t-d\right\vert ^{\beta-\mathbf{p}+1}-\left\vert
a-d\right\vert ^{\beta-\mathbf{p}+1}}{\beta-\mathbf{p}+1}\right\} \nonumber\\
&  =\tfrac{\Gamma\left(  \beta+1\right)  }{\Gamma\left(  \beta-\mathbf{p}%
+2\right)  \Gamma\left(  \alpha+\mathbf{p}\right)  }\left\{  \frac{\left\vert
t-a\right\vert ^{\alpha}}{\left\vert t-d\right\vert ^{-\beta}}\left\vert
\frac{t-a}{t-d}\right\vert ^{\mathbf{p}-1}-\frac{\left\vert t-a\right\vert
^{\alpha}}{\left\vert a-d\right\vert ^{-\beta}}\left\vert \frac{t-a}%
{a-d}\right\vert ^{\mathbf{p}-1}\right\}  . \label{IP3}%
\end{align}
\newline While for the second factor in Eq.(\ref{IP3}) it is straightforward
that
\begin{equation}
\lim_{\mathbf{p}\rightarrow\infty}\left\{  \frac{\left\vert t-a\right\vert
^{\alpha}}{\left\vert t-d\right\vert ^{-\beta}}\left\vert \frac{t-a}%
{t-d}\right\vert ^{\mathbf{p}-1}-\frac{\left\vert t-a\right\vert ^{\alpha}%
}{\left\vert a-d\right\vert ^{-\beta}}\left\vert \frac{t-a}{a-d}\right\vert
^{\mathbf{p}-1}\right\}  =0, \label{IP4}%
\end{equation}
whenever $t\in\left[  a,a+\frac{\left\vert a-d\right\vert }{2}\right)  $ if
$a\leq t<d$ or $t\in\left[  a,a+\left\vert a-d\right\vert \right)  $ if
$d<a\leq t$, for the first factor, we need some further analysis. Using the
well known identity \cite{Magnus}%
\begin{equation}
\Gamma\left(  z-n\right)  =\frac{\left(  -1\right)  ^{n}\pi}{\sin(\pi
z)\Gamma\left(  n+1-z\right)  }, \label{IP5}%
\end{equation}
we can write%
\begin{align}
\left\vert \frac{\Gamma\left(  \beta+1\right)  }{\Gamma\left(  \beta
-\mathbf{p}+2\right)  \Gamma\left(  \alpha+\mathbf{p}\right)  }\right\vert  &
=\left\vert \frac{\Gamma\left(  \beta+1\right)  \sin[\pi(\beta+2)]\Gamma
\left(  \mathbf{p}-\beta-1\right)  }{\left(  -1\right)  ^{p}\pi\Gamma\left(
\alpha+\mathbf{p}\right)  }\right\vert \nonumber\\
&  \leq\left\vert \frac{\Gamma\left(  \beta+1\right)  \sin[\pi(\beta
+2)]}{\left(  -1\right)  ^{p}\pi}\right\vert \left\vert \frac{\Gamma\left(
\mathbf{p}-\beta-1\right)  }{\Gamma\left(  \alpha+\mathbf{p}\right)
}\right\vert \nonumber\\
&  =\left\vert \frac{\Gamma\left(  \beta+1\right)  \sin(\pi\beta)}{\pi
}\right\vert \left\vert \frac{\Gamma\left(  \mathbf{p}-\beta-1\right)
}{\Gamma\left(  \alpha+\mathbf{p}\right)  }\right\vert . \label{IP6}%
\end{align}
\ For our choice of $\beta$, it is secured that $\left\vert \frac
{\Gamma\left(  \beta+1\right)  \sin(\pi\beta)}{\pi}\right\vert =M$ is always
finite regardless of $\mathbf{p}$, while from \cite{Tricomi} we have the
asymptotic behavior of a ratio of gamma functions%
\begin{equation}
\lim_{\mathbf{p}\rightarrow\infty}\left\vert \frac{\Gamma\left(
\mathbf{p}-\beta-1\right)  }{\Gamma\left(  \alpha+\mathbf{p}\right)
}\right\vert =\mathbf{p}^{-\alpha-\beta-1}, \label{IP7}%
\end{equation}
with convergence depending on $-\alpha-1\leq\beta$. However, we are actually
interested on the behavior of the product of the terms in Eq.(\ref{IP4}) and
Eq.(\ref{IP7}). Since%
\begin{align}
\left\vert \frac{t-a}{t-d}\right\vert ^{\mathbf{p}-1}  &  =e^{-\gamma
\mathbf{p}+\gamma},\label{IP8a}\\
\left\vert \frac{t-a}{a-d}\right\vert ^{\mathbf{p}-1}  &  =e^{-\eta
\mathbf{p}+\eta}, \label{IP8b}%
\end{align}
with $0<\gamma=-\ln\left\vert \frac{t-a}{t-d}\right\vert $ and $0<\eta
=-\ln\left\vert \frac{t-a}{a-d}\right\vert $ (for the proper neighborhood as
describe above) and the exponential decay (or growth) rate is always stronger
(in the limit) than any power like rate of growth (decay), that is%
\begin{equation}
\lim_{\mathbf{p}\rightarrow\infty}\frac{\mathbf{p}^{a}}{e^{\mathbf{p}}%
}=0,\text{ }\forall a\in%
\mathbb{R}
, \label{IP9}%
\end{equation}
we can conclude that%
\begin{equation}
\lim_{\mathbf{p}\rightarrow\infty}\left\vert \mathcal{R}_{\mathbf{p}%
}\right\vert \leq M\lim_{\mathbf{p}\rightarrow\infty}\mathbf{p}^{-\alpha
-\beta-1}\left\{  e^{-\gamma\mathbf{p}+\gamma}-e^{-\eta\mathbf{p}+\eta
}\right\}  =0, \label{IP10}%
\end{equation}
proving the results in Eq.(\ref{Ia}) and Eq.(\ref{Ib}) for $\beta\in%
\mathbb{R}
\setminus%
\mathbb{Z}
$.

We now investigate the case $\beta=-m$, $m\in%
\mathbb{N}
$. Initially, we will exclude the case $m=1$, since they lead to logarithms.
Recall that for $\beta=-2,-3,...$ the domain of $f$ is
\[
Dom(f)=%
\mathbb{R}
\setminus\{d\}=\left(  -\infty,d\right)  \cup\left(  d,\infty\right)  .
\]
Therefore, when calculating $\left[  _{a}J_{t}^{\alpha}\left(  x-d\right)
^{-m}\right]  (t)$, we need to take care of choosing the lower limit of
integration $a$ in one of the following intervals: (I) $\left(  -\infty
,d\right)  $ or (II) $\left(  d,\infty\right)  $.\newline(I) So let $a\leq
t<d$. Then,%
\begin{equation}
\left[  _{a}J_{t}^{\alpha}\left(  x-d\right)  ^{-m}\right]  (t)=\frac
{1}{\Gamma\left(  \alpha\right)  }\int_{a}^{t}\left(  x-d\right)  ^{-m}\left(
t-x\right)  ^{\alpha-1}dx. \label{15}%
\end{equation}
If we integrate by parts the expression in the right side of Eq.(\ref{15}) a
total of $\mathbf{p}$ times we get the following result%
\begin{equation}
\left[  _{a}J_{t}^{\alpha}\left(  x-d\right)  ^{-m}\right]  (t)=\sum
_{k=0}^{\mathbf{p}-1}\frac{\left(  -1\right)  ^{k}\left(  m\right)
_{k}\left(  a-d\right)  ^{-\left(  m+k\right)  }\left(  t-a\right)
^{\alpha+k}}{\Gamma\left(  \alpha+k+1\right)  }+\mathcal{R}_{\mathbf{p}},
\label{16a}%
\end{equation}
where
\begin{equation}
\mathcal{R}_{\mathbf{p}}=\frac{\left(  -1\right)  ^{\mathbf{p}}\left(
m\right)  _{\mathbf{p}}}{\Gamma\left(  \alpha+\mathbf{p}\right)  }\int_{a}%
^{t}\left(  x-d\right)  ^{-\left(  m+\mathbf{p}\right)  }\left(  t-x\right)
^{\alpha+\mathbf{p}-1}dx. \label{16b}%
\end{equation}
\ Using the identity in Eq.(\ref{IdA}), the expressions in Eq.(\ref{16a}) and
Eq.(\ref{16b}) can be rewritten as%
\begin{equation}
\left[  _{a}J_{t}^{\alpha}\left(  x-d\right)  ^{-m}\right]  (t)=\sum
_{k=0}^{\mathbf{p}-1}\frac{\left(  -1\right)  ^{k}\Gamma\left(  m+k\right)
\left(  a-d\right)  ^{-\left(  m+k\right)  }\left(  t-a\right)  ^{\alpha+k}%
}{\Gamma\left(  m\right)  \Gamma\left(  \alpha+k+1\right)  }+\mathcal{R}%
_{\mathbf{p}}, \label{18a}%
\end{equation}
where%
\begin{equation}
\mathcal{R}_{\mathbf{p}}=\frac{\left(  -1\right)  ^{\mathbf{p}}\Gamma\left(
m+\mathbf{p}\right)  }{\Gamma\left(  m\right)  \Gamma\left(  \alpha
+\mathbf{p}\right)  }\int_{a}^{t}\left(  x-d\right)  ^{-\left(  m+\mathbf{p}%
\right)  }\left(  t-x\right)  ^{\alpha+\mathbf{p}-1}dx. \label{18b}%
\end{equation}
Observe that the remainder $\mathcal{R}_{\mathbf{p}}$ can be estimated by the
inequalities%
\begin{align}
\left\vert \mathcal{R}_{\mathbf{p}}\right\vert  &  \leq\left\vert
\frac{\left(  -1\right)  ^{\mathbf{p}}\Gamma\left(  m+\mathbf{p}\right)
}{\Gamma\left(  m\right)  \Gamma\left(  \alpha+\mathbf{p}\right)  }\right\vert
\int_{a}^{t}\left\vert \left(  x-d\right)  \right\vert ^{-\left(
m+\mathbf{p}\right)  }\left\vert \left(  t-x\right)  \right\vert
^{\alpha+\mathbf{p}-1}dx\nonumber\\
&  \leq\left\vert \frac{\Gamma\left(  m+\mathbf{p}\right)  }{\Gamma\left(
m\right)  \Gamma\left(  \alpha+\mathbf{p}\right)  }\right\vert \left\vert
t-a\right\vert ^{\alpha+\mathbf{p}-1}\int_{a}^{t}\left\vert \left(
x-d\right)  \right\vert ^{-\left(  m+\mathbf{p}\right)  }dx\nonumber\\
&  \leq\left\vert \frac{\Gamma\left(  m+\mathbf{p}\right)  }{\Gamma\left(
m\right)  \Gamma\left(  \alpha+\mathbf{p}\right)  }\right\vert \left\vert
t-a\right\vert ^{\alpha+\mathbf{p}-1}\int_{a}^{t}\left\vert \left(
t-d\right)  \right\vert ^{-\left(  m+\mathbf{p}\right)  }dx\nonumber\\
&  =\left\vert \frac{\Gamma\left(  m+\mathbf{p}\right)  }{\Gamma\left(
m\right)  \Gamma\left(  \alpha+\mathbf{p}\right)  }\right\vert \frac
{\left\vert t-a\right\vert ^{\alpha+\mathbf{p}}}{\left\vert t-d\right\vert
^{m+\mathbf{p}}}\nonumber\\
&  =\left\vert \frac{\Gamma\left(  m+\mathbf{p}\right)  }{\Gamma\left(
m\right)  \Gamma\left(  \alpha+\mathbf{p}\right)  }\right\vert \frac
{\left\vert t-a\right\vert ^{\alpha}}{\left\vert t-d\right\vert ^{m}%
}\left\vert \frac{t-a}{t-d}\right\vert ^{\mathbf{p}}. \label{19}%
\end{align}
Now, in the last equality of Eq.(\ref{19}), we have that for each fixed value
of $t$ satisfying $a\leq t<d$ the fraction $\frac{\left\vert t-a\right\vert
^{\alpha}}{\left\vert t-d\right\vert ^{m}}$ has a fixed finite value. While%
\begin{equation}
\lim_{\mathbf{p}\rightarrow\infty}\left\vert \frac{t-a}{t-d}\right\vert
^{\mathbf{p}}=0, \label{20a}%
\end{equation}
as long as $\left\vert \frac{t-a}{t-d}\right\vert <1$, which is guaranteed for
$t\in\left[  a,a+\frac{\epsilon}{2}\right)  $, $\epsilon=\left\vert
d-a\right\vert $. On the other hand, we have that \cite{Tricomi}%
\begin{equation}
\lim_{\mathbf{p}\rightarrow\infty}\left\vert \frac{\Gamma\left(
m+\mathbf{p}\right)  }{\Gamma\left(  m\right)  \Gamma\left(  \alpha
+\mathbf{p}\right)  }\right\vert =\frac{\mathbf{p}^{m-\alpha}}{\Gamma\left(
m\right)  }, \label{20b}%
\end{equation}
with convergence depending on $m-\alpha\leq0$. However, again we are basically
concerned with the limit of the product between Eq.(\ref{20a}) and
Eq.(\ref{20b}) and since we can identify%
\begin{equation}
\left\vert \frac{t-a}{t-d}\right\vert ^{\mathbf{p}}=e^{-\gamma\mathbf{p}},
\label{20c}%
\end{equation}
where $0<\gamma=-\ln\left\vert \frac{t-a}{t-d}\right\vert $ (for $t\in\left[
a,a+\frac{\epsilon}{2}\right)  $, $\epsilon=\left\vert d-a\right\vert $), and
the exponential decay rate is faster (in the limit) than any power like rate,
then we can conclude that%
\begin{equation}
\lim_{p\rightarrow\infty}\left\vert \mathcal{R}_{\mathbf{p}}\right\vert
\leq\frac{\left\vert t-a\right\vert ^{\alpha}}{\Gamma\left(  m\right)
\left\vert t-d\right\vert ^{m}}\lim_{p\rightarrow\infty}\frac{\mathbf{p}%
^{m-\alpha}}{e^{\gamma\mathbf{p}}}=0, \label{20d}%
\end{equation}
therefore, as $\mathbf{p}\rightarrow\infty$, we conclude that Eq.(\ref{18a})
reduces to%
\begin{equation}
\left[  _{a}J_{t}^{\alpha}\left(  x-d\right)  ^{-m}\right]  (t)=\sum
_{k=0}^{\infty}\frac{\left(  -1\right)  ^{k}\Gamma\left(  m+k\right)  \left(
a-d\right)  ^{-\left(  m+k\right)  }\left(  t-a\right)  ^{\alpha+k}}%
{\Gamma\left(  m\right)  \Gamma\left(  \alpha+k+1\right)  }, \label{21}%
\end{equation}
for $t\in\left[  a,a+\frac{\epsilon}{2}\right)  $, $\epsilon=\left\vert
d-a\right\vert $, giving the expression claimed in Eq.(\ref{Id}). A very
similar analysis but with some minor adjustment on the neighborhood would
prove Eq.(\ref{Ie}) as well.\newline It remains to explore the case where
$m=1$, that is when $\beta=-1$. For that, recall from integer order calculus
that integration of functions of the type $\left(  x-d\right)  ^{-1}$ lead to
logarithms. Specifically we have%
\begin{equation}
\int_{a}^{t}\left(  x-d\right)  ^{-1}dx=\ln\left(  \frac{t-d}{a-d}\right)
,\text{ }a>d\text{.} \label{25}%
\end{equation}
But we can also calculate the above integral in the following way: We consider
the power series representation of $f(t)=\left(  t-d\right)  ^{-1}$ centered
at $t=a$ (for $a>d$),%
\begin{equation}
\left(  t-d\right)  ^{-1}=-\sum_{k=0}^{\infty}\frac{\left(  t-a\right)  ^{k}%
}{\left(  d-a\right)  ^{k+1}}=\sum_{k=0}^{\infty}\frac{\left(  -1\right)
^{k}\left(  t-a\right)  ^{k}}{\left(  a-d\right)  ^{k+1}} \label{26}%
\end{equation}
and then integrate it (inside its radius of convergence) to get%
\begin{align}
\int_{a}^{t}\left(  x-d\right)  ^{-1}dx  &  =\int_{a}^{t}-\sum_{k=0}^{\infty
}\frac{\left(  t-a\right)  ^{k}}{\left(  d-a\right)  ^{k+1}}dx\nonumber\\
&  =-\sum_{k=0}^{\infty}\frac{1}{\left(  d-a\right)  ^{k+1}}\frac{\left(
t-a\right)  ^{k+1}}{k+1}\label{27a}\\
&  =\sum_{k=0}^{\infty}\frac{\left(  -1\right)  ^{k}}{k+1}\left(  \frac
{t-a}{a-d}\right)  ^{k+1}, \label{27b}%
\end{align}
where either of the last two series are valid representations for $\ln\left(
\frac{t-d}{a-d}\right)  $.\newline Now lets consider the explicit fractional
case $0<\alpha<1$. By definition, we have%
\begin{equation}
\left[  _{a}J_{t}^{\alpha}\left(  x-d\right)  ^{-1}\right]  (t)=\frac
{1}{\Gamma\left(  \alpha\right)  }\int_{a}^{t}\left(  x-d\right)  ^{-1}\left(
t-x\right)  ^{\alpha-1}dx, \label{28}%
\end{equation}
proceeding with integration by parts in the same way as the previous cases,
one get the result%
\begin{equation}
\left[  _{a}J_{t}^{\alpha}\left(  x-d\right)  ^{-1}\right]  (t)=\sum
_{k=0}^{\mathbf{\infty}}\frac{\left(  -1\right)  ^{k}\Gamma\left(  1+k\right)
\epsilon^{-\left(  1+k\right)  }\left(  t-d^{+}\right)  ^{\alpha+k}}%
{\Gamma\left(  \alpha+k+1\right)  }, \label{29}%
\end{equation}
for $t\in\left[  d^{+},d^{+}+\epsilon\right)  $, which is nothing else but the
formula in Eq.(\ref{Ie}) when we set $m=1$.\newline To recover the original
expressions in Eq.(\ref{Ia}) and Eq.(\ref{Ib}) we can simply verify that since
$\beta=-m$, then%
\[
\frac{\left(  -1\right)  ^{k}\Gamma\left(  m+k\right)  }{\Gamma\left(
m\right)  }=\frac{\left(  -1\right)  ^{k}\Gamma\left(  -\beta+k\right)
}{\Gamma\left(  -\beta\right)  },
\]
and by Eq.(\ref{IdAA}) together with the identities in Eq.(\ref{IdA}) and
Eq.(\ref{IdB}), we have that%
\[
\frac{\left(  -1\right)  ^{k}\Gamma\left(  -\beta+k\right)  }{\Gamma\left(
-\beta\right)  }=\left(  -1\right)  ^{k}\left(  -\beta\right)  _{k}=\left(
\beta\right)  _{-k}=\frac{\Gamma\left(  \beta+1\right)  }{\Gamma\left(
\beta-k+1\right)  }.
\]
\newline Finally, for the simplest case $\beta=m\in%
\mathbb{N}
$, $Dom(f)=%
\mathbb{R}
$, so it really doesn't matter where we take the lower limit $t=a$, so%
\begin{align}
\left[  _{a}J_{t}^{\alpha}\left(  x-d\right)  ^{m}\right]  (t)  &  =\frac
{1}{\Gamma\left(  \alpha\right)  }\int_{a}^{t}(x-d)^{m}\left(  t-x\right)
^{\alpha-1}dx\nonumber\\
&  =\text{after }m\text{ integration by parts}\nonumber\\
&  =%
{\displaystyle\sum\limits_{k=0}^{m}}
\frac{\left(  m\right)  _{-k}\left(  a-d\right)  ^{m-k}\left(  t-a\right)
^{\alpha+k}}{\Gamma\left(  \alpha+k+1\right)  }\nonumber\\
&  =%
{\displaystyle\sum\limits_{k=0}^{m}}
\frac{\Gamma\left(  m+1\right)  \left(  a-d\right)  ^{m-k}\left(  t-a\right)
^{\alpha+k}}{\Gamma\left(  m-k+1\right)  \Gamma\left(  \alpha+k+1\right)
},\text{ }t\in%
\mathbb{R}
, \label{7}%
\end{align}
where we have made use of the descending Pochhammer symbol Eq.(\ref{PocD}) and
the identity in Eq.(\ref{IdB}).
\end{proof}

\begin{corollary}
\label{Cor}Consider the hypothesis of \emph{Theorem \ref{Th1}}. If the lower
limit $t=a\in U\subset\Omega$, where $U$ is an interval where the power
function is analytic, then calculating its \emph{RLFI} can be done by
operating on its Taylor series expansions.
\end{corollary}

\begin{proof}
Just consider the Taylor series expansion of $f(t)=\left(  t-d\right)
^{\beta}$ at $t=a$ and operate $_{a}J_{t}^{\alpha}$ on it. Due to the uniform
convergence, we can integrate term by term obtaining the expressions listed on
Theorem \ref{Th1}.
\end{proof}

\begin{remark}
\label{R1}It is easily verifiable, that each of the expressions in
\emph{Eq.(\ref{Ia})} - \emph{Eq.(\ref{Ie})}, reduces to the expected integer
order expressions when choosing $\alpha=0$ or $\alpha=1$ and, as an example,
we will show this for \emph{Eq.(\ref{Ia})} \emph{(}the others are
similar\emph{)}. Indeed, setting $\alpha=0$ on \emph{Eq.(\ref{Ia})} we
directly get the Taylor series expansion of $f(t)=\left(  t-d\right)  ^{\beta
}$ centered on $t=a$, while setting $\alpha=1$ on \emph{Eq.(\ref{Ia})} we get%
\begin{equation}
\left[  _{a}J_{t}\left(  x-d\right)  ^{\beta}\right]  (t)=\sum_{k=0}^{\infty
}\frac{\Gamma\left(  \beta+1\right)  \left(  a-d\right)  ^{\beta-k}\left(
t-a\right)  ^{k+1}}{\Gamma\left(  \beta-k+1\right)  \Gamma\left(  k+2\right)
}. \label{E1}%
\end{equation}
\newline On the other hand, since the power function is analytic in the
interval in consideration, then%
\begin{align}
\frac{\left(  t-d\right)  ^{\beta+1}-\left(  a-d\right)  ^{\beta+1}}{\beta+1}
&  =\int_{a}^{t}\left(  x-d\right)  ^{\beta}dx\nonumber\\
&  =\int_{a}^{t}\sum_{k=0}^{\infty}\left(
\begin{array}
[c]{c}%
\beta\\
k
\end{array}
\right)  \left(  a-d\right)  ^{\beta-k}\left(  x-a\right)  ^{k}dx\nonumber\\
&  =\sum_{k=0}^{\infty}\left(
\begin{array}
[c]{c}%
\beta\\
k
\end{array}
\right)  \left(  a-d\right)  ^{\beta-k}\int_{a}^{t}\left(  x-a\right)
^{k}dx\nonumber\\
&  =\sum_{k=0}^{\infty}\left(
\begin{array}
[c]{c}%
\beta\\
k
\end{array}
\right)  \left(  a-d\right)  ^{\beta-k}\frac{\left(  t-a\right)  ^{k+1}}{k+1},
\label{E2}%
\end{align}
and clearly, \emph{Eq.(\ref{E1})} equals \emph{Eq.(\ref{E2})} since $\left(
k+1\right)  \Gamma\left(  k+1\right)  =\Gamma\left(  k+2\right)  $.
\end{remark}

\qquad Before ending this section, there's one final observation that we want
to call attention. All expressions listed on Theorem \ref{Th1} for the power
functions $f(t)=\left(  t-d\right)  ^{\beta}$ are only valid when we choose
the lower limit $a\neq d$, which guarantees the convergence of the series in
their respective intervals of definition ($\left[  a,a+\frac{\epsilon}%
{2}\right)  $ or $\left[  d^{+},d^{+}+\epsilon\right)  $, with $\epsilon
=\left\vert d-a\right\vert >0$). The only exception to this is the case when
$\beta=m\in%
\mathbb{N}
_{0}$, where if we set $a=d$ (that means $\epsilon=0$) in Eq.(\ref{Ic}) it
reduces to%
\begin{equation}
\left[  _{a}J_{t}^{\alpha}\left(  x-d\right)  ^{m}\right]  (t)=\frac
{\Gamma\left(  m+1\right)  \left(  t-a\right)  ^{\alpha+m}}{\Gamma\left(
\alpha+m+1\right)  },\text{ }a,t\in%
\mathbb{R}
, \label{E3}%
\end{equation}
which agrees with the usual formula in Eq.(\ref{Ipower}). But this is to be
expected. In one hand, Eq.(\ref{Ic}) is a finite sum and its convergence
doesn't depend on $\epsilon$, on the other hand we known that the only class
of power functions that are analytic everywhere in its domain of definition
(which in such case includes the point $t=d$) are the polynomials. For all
other values of the index $\beta$, these functions are not analytic at $t=d$
even if $t=d$ belongs to its domain (e.g., consider $f(t)=\left(  t-d\right)
^{\frac{1}{2}}$).

\section{Riemann-Liouville Differentiation of Power Functions}

The result for the RLFD comes as a corollary of \textbf{Theorem \ref{Th1}}.

\begin{corollary}
\label{Cor1}Let $f(t)=\left(  t-d\right)  ^{\beta}$, $d,\beta\in%
\mathbb{R}
$ and suppose $t\in\Omega\subset%
\mathbb{R}
$, where $\Omega$ is an interval where $f$ is properly defined as real valued
function. Then\newline\emph{(I)}
\begin{equation}
\left[  _{a}D_{t}^{\alpha}\left(  x-d\right)  ^{\beta}\right]  (t)=\sum
_{k=0}^{\mathbf{\infty}}\frac{\Gamma\left(  \beta+1\right)  \left(
a-d\right)  ^{\beta-k}\left(  t-a\right)  ^{-\alpha+k}}{\Gamma\left(
\beta-k+1\right)  \Gamma\left(  -\alpha+k+1\right)  },\text{ }t\in\left[
a,a+\frac{\epsilon}{2}\right)  \label{Da}%
\end{equation}
\newline\emph{(II)}
\begin{equation}
\left[  _{d^{+}}D_{t}^{\alpha}\left(  x-d\right)  ^{\beta}\right]
(t)=\sum_{k=0}^{\mathbf{\infty}}\frac{\Gamma\left(  \beta+1\right)
\epsilon^{\beta-k}\left(  t-d^{+}\right)  ^{-\alpha+k}}{\Gamma\left(
\beta-k+1\right)  \Gamma\left(  -\alpha+k+1\right)  },\text{ }t\in\left[
d^{+},d^{+}+\epsilon\right)  . \label{Db}%
\end{equation}
\newline Particularly, if $\beta=m\in%
\mathbb{N}
_{0}$, then%
\begin{equation}
\left[  _{a}D_{t}^{\alpha}\left(  x-d\right)  ^{m}\right]  (t)=\sum_{k=0}%
^{m}\frac{\Gamma\left(  m+1\right)  \left(  a-d\right)  ^{m-k}\left(
t-a\right)  ^{k-\alpha}}{\Gamma\left(  m-k+1\right)  \Gamma\left(
k+1-\alpha\right)  },\text{ }a,t\in%
\mathbb{R}
, \label{Dc}%
\end{equation}
while if $\beta=-m$ \ with $m\in%
\mathbb{N}
$, then we can use alternatively the following expressions as well%
\begin{align}
\left[  _{a}D_{t}^{\alpha}\left(  x-d\right)  ^{-m}\right]  (t)  &
=\sum_{k=0}^{\mathbf{\infty}}\frac{\left(  -1\right)  ^{k}\Gamma\left(
m+k\right)  \left(  a-d\right)  ^{-m-k}\left(  t-a\right)  ^{-\alpha+k}%
}{\Gamma\left(  m\right)  \Gamma\left(  -\alpha+k+1\right)  },\label{47aa}\\
\left[  _{d^{+}}D_{t}^{\alpha}\left(  x-d\right)  ^{-m}\right]  (t)  &
=\sum_{k=0}^{\mathbf{\infty}}\frac{\left(  -1\right)  ^{k}\Gamma\left(
m+k\right)  \epsilon^{-\left(  m+k\right)  }\epsilon^{-m-k}\left(
t-d^{+}\right)  ^{-\alpha+k}}{\Gamma\left(  m\right)  \Gamma\left(
-\alpha+k+1\right)  }, \label{47bb}%
\end{align}
with $t\in\left[  a,a+\frac{\epsilon}{2}\right)  $ and $t\in\left[
d^{+},d^{+}+\epsilon\right)  $, respectively and where in all cases,
$\epsilon=\left\vert d-a\right\vert >0.$
\end{corollary}

\begin{proof}
We can use a similar tedious analysis as in Theorem \ref{Th1} or simply
realize that by the definition of the RLFD for $0<\alpha<1$ (refer to
Eq.(\ref{2})) we have%
\begin{equation}
\left[  _{a}D_{t}^{\alpha}\left(  x-d\right)  ^{\beta}\right]  (t)=\frac
{d}{dt}\left[  _{a}J_{t}^{1-\alpha}\left(  x-d\right)  ^{\beta}\right]  (t),
\label{PD1}%
\end{equation}
then it suffice to differentiate the corresponding expressions obtained for
$\left[  _{a}J_{t}^{1-\alpha}\left(  x-d\right)  ^{\beta}\right]  (t)$ using
the results of Theorem \ref{Th1}.
\end{proof}

\begin{corollary}
Consider the hypothesis of \emph{Corollary \ref{Cor1}}. If the lower limit
$t=a\in U\subset\Omega$, where $U$ is an interval where the power function is
analytic, then calculating its \emph{RLFD} can be done by operating on its
Taylor series expansions.
\end{corollary}

\begin{proof}
Same as Corollary \ref{Cor}.
\end{proof}

\begin{remark}
The expressions in \emph{Eq.(\ref{Da})} - \emph{Eq.(\ref{Db})} also reduce to
the expected integer order formulas when setting $\alpha=0$ and $\alpha=1$ and
the proofs can be done in a similar way as in \emph{Remark \ref{R1}}.
\end{remark}

\section{Summarizing the Results}

We have then the following expressions for the RLFI of order $0<\alpha<1$ for
power functions of any order.%
\begin{align*}
\left[  _{a}J_{t}^{\alpha}\left(  x-d\right)  ^{m}\right]  (t)  &  =%
{\displaystyle\sum\limits_{k=0}^{m}}
\frac{\Gamma\left(  m+1\right)  \left(  a-d\right)  ^{m-k}\left(  t-a\right)
^{\alpha+k}}{\Gamma\left(  m-k+1\right)  \Gamma\left(  \alpha+k+1\right)
},\text{ }m\in%
\mathbb{N}
_{0},\text{ }t\in%
\mathbb{R}
,\\
\left[  _{a}J_{t}^{\alpha}\left(  x-d\right)  ^{\beta}\right]  (t)  &
=\sum_{k=0}^{\infty}\frac{\Gamma\left(  \beta+1\right)  \left(  a-d\right)
^{\beta-k}\left(  t-a\right)  ^{\alpha+k}}{\Gamma\left(  \beta-k+1\right)
\Gamma\left(  \alpha+k+1\right)  },\text{ }\beta\in%
\mathbb{R}
\setminus%
\mathbb{N}
_{0},\text{ }t\in\Omega_{1},\\
\left[  _{d^{+}}J_{t}^{\alpha}\left(  x-d\right)  ^{\beta}\right]  (t)  &
=\sum_{k=0}^{\mathbf{\infty}}\frac{\Gamma\left(  \beta+1\right)
\epsilon^{\beta-k}\left(  t-d^{+}\right)  ^{\alpha+k}}{\Gamma\left(
\beta-k+1\right)  \Gamma\left(  \alpha+k+1\right)  },\text{ }\beta\in%
\mathbb{R}
\setminus%
\mathbb{N}
_{0},\text{ }t\in\Omega_{2},
\end{align*}
where $\Omega_{1}=\left[  a,a+\frac{\left\vert d-a\right\vert }{2}\right)  $
and $\Omega_{2}=\left[  d^{+},d^{+}+\epsilon\right)  $. Particularly, if
$\beta=-m$ \ with $m\in%
\mathbb{N}
$, then we can use alternatively these expressions as well
\begin{align*}
\left[  _{a}J_{t}^{\alpha}\left(  x-d\right)  ^{-m}\right]  (t)  &
=\sum_{k=0}^{\infty}\frac{\left(  -1\right)  ^{k}\Gamma\left(  m+k\right)
\left(  a-d\right)  ^{-\left(  m+k\right)  }\left(  t-a\right)  ^{\alpha+k}%
}{\Gamma\left(  m\right)  \Gamma\left(  \alpha+k+1\right)  },\text{ }%
t\in\Omega_{1},\\
\left[  _{d^{+}}J_{t}^{\alpha}\left(  x-d\right)  ^{-m}\right]  (t)  &
=\sum_{k=0}^{\mathbf{\infty}}\frac{\left(  -1\right)  ^{k}\Gamma\left(
m+k\right)  \epsilon^{-\left(  m+k\right)  }\left(  t-d^{+}\right)
^{\alpha+k}}{\Gamma\left(  m\right)  \Gamma\left(  \alpha+k+1\right)  },\text{
}t\in\Omega_{2}.
\end{align*}

Now for the RLFD of order $0<\alpha<1$ for power functions of any order, we
have%
\begin{align*}
\left[  _{a}D_{t}^{\alpha}\left(  x-d\right)  ^{m}\right]  (t)  &  =\sum
_{k=0}^{m}\frac{\Gamma\left(  m+1\right)  \left(  a-d\right)  ^{m-k}\left(
t-a\right)  ^{k-\alpha}}{\Gamma\left(  m-k+1\right)  \Gamma\left(
k+1-\alpha\right)  },m\in%
\mathbb{N}
_{0},\text{ }t\in%
\mathbb{R}
,\\
\left[  _{a}D_{t}^{\alpha}\left(  x-d\right)  ^{\beta}\right]  (t)  &
=\sum_{k=0}^{\mathbf{\infty}}\frac{\Gamma\left(  \beta+1\right)  \left(
a-d\right)  ^{\beta-k}\left(  t-a\right)  ^{-\alpha+k}}{\Gamma\left(
\beta-k+1\right)  \Gamma\left(  -\alpha+k+1\right)  },\text{ }\beta\in%
\mathbb{R}
\setminus%
\mathbb{N}
_{0},\text{ }t\in\Omega_{1},\\
\left[  _{d^{+}}D_{t}^{\alpha}\left(  x-d\right)  ^{\beta}\right]  (t)  &
=\sum_{k=0}^{\mathbf{\infty}}\frac{\Gamma\left(  \beta+1\right)
\epsilon^{\beta-k}\left(  t-d^{+}\right)  ^{-\alpha+k}}{\Gamma\left(
\beta-k+1\right)  \Gamma\left(  -\alpha+k+1\right)  },\text{ }\beta\in%
\mathbb{R}
\setminus%
\mathbb{N}
_{0},\text{ }t\in\Omega_{2},
\end{align*}
particularly, if $\beta=-m$ \ with $m\in%
\mathbb{N}
$, then we can use alternatively these expressions as well%
\begin{align*}
\left[  _{a}D_{t}^{\alpha}\left(  x-d\right)  ^{-m}\right]  (t)  &
=\sum_{k=0}^{\mathbf{\infty}}\frac{\left(  -1\right)  ^{k}\Gamma\left(
m+k\right)  \left(  a-d\right)  ^{-\left(  m+k\right)  }\left(  t-a\right)
^{-\alpha+k}}{\Gamma\left(  m\right)  \Gamma\left(  -\alpha+k+1\right)
},\text{ }t\in\Omega_{1},\\
\left[  _{d^{+}}D_{t}^{\alpha}\left(  x-d\right)  ^{\beta}\right]  (t)  &
=\sum_{k=0}^{\mathbf{\infty}}\frac{\left(  -1\right)  ^{k}\Gamma\left(
m+k\right)  \epsilon^{-\left(  m+k\right)  }\left(  t-d^{+}\right)
^{-\alpha+k}}{\Gamma\left(  m\right)  \Gamma\left(  -\alpha+k+1\right)
},\text{ }t\in\Omega_{2}.
\end{align*}

Finally, we point out that in the appendix we have related the expressions for
$\left[  _{a}J_{t}^{\alpha}\left(  x-d\right)  ^{\beta}\right]  (t)$, $\left[
_{d^{+}}J_{t}^{\alpha}\left(  x-d\right)  ^{\beta}\right]  (t)$, $\left[
_{a}D_{t}^{\alpha}\left(  x-d\right)  ^{\beta}\right]  (t)$ and $\left[
_{d^{+}}D_{t}^{\alpha}\left(  x-d\right)  ^{\beta}\right]  (t)$ in terms of
the hypergeometric functions.

\section{Conclusion}

In this work, we have shown that it is possible to calculate the RLFI and RLFD
of order $0<\alpha<1$ of power functions $\left(  t-\ast\right)  ^{\beta}$
with $\beta$ being any real value and we are able to express the results in
terms of function series of the type $\left(  t-\ast\right)  ^{\pm\alpha+k}$
and that such expressions can be related to the famous hypergeometric
functions of the Mathematical-Physics. We have also observed that since the
Riemann-Liouville formulations of the fractional integral and differential
operators are actually defined in terms of the notion of a \emph{definite
integral}, then the series obtained are convergent on a proper neighborhood of
the lower limit, therefore careful attention must be taken to check if the
lower limit belongs (or not) to the original function's domain. In fact, the
whole problem of \textquotedblleft not being able to integrodifferentiate the
power functions of index strictly less than $-1$\textquotedblright\ is not
really distinct than being (or not) able to integrate ordinarily power
functions with a singularity in the lower limit. Recall that in ordinary
calculus (integer orders), the integral's final result of achieving or not a
valid expression actually depends on \textquotedblleft how strong is the
singularity in the lower limit\textquotedblright\ versus how large is the
order of the $n$-fold integral, so for the sake of instigating future
discussions and works, this strongly suggest further investigation when
setting the order $\alpha$ with values greater than unity. It is also
interesting to look if one can obtain a better formulation of the notions of
an $\alpha$-primitive and therefore something as a \emph{fractional indefinite
integral of order }$\alpha$.

\section{Acknowledgment}

F. G. Rodrigues would like to acknowledge the support from CNPq under grant 200832/2015-8.

\appendix

\section{Appendix}

We will show that following some algebraic manipulations and use of identities
we can write the expressions found for the RLFI and RLFD also in terms of the
hypergeometric functions. For starters, we will describe the steps for
Eq.(\ref{Ib}) and the others are done similarly.

So we have that%
\begin{align*}
\left[  _{d^{+}}J_{t}^{\alpha}\left(  x-d\right)  ^{\beta}\right]  (t)  &
=\sum_{k=0}^{\mathbf{\infty}}\frac{\Gamma\left(  \beta+1\right)
\epsilon^{\beta-k}\left(  t-d^{+}\right)  ^{\alpha+k}}{\Gamma\left(
\beta-k+1\right)  \Gamma\left(  \alpha+k+1\right)  }\\
&  =\Gamma\left(  \beta+1\right)  \epsilon^{\alpha+\beta}\left(  \frac
{t-d}{\epsilon}-1\right)  ^{\alpha}\sum_{k=0}^{\mathbf{\infty}}\frac{\left(
\frac{t-d}{\epsilon}-1\right)  ^{k}}{\Gamma\left(  \beta-k+1\right)
\Gamma\left(  \alpha+k+1\right)  }.
\end{align*}

To simplify the notation we introduce $z=\frac{t-d}{\epsilon}>0$ and $\left[
_{d^{+}}J_{t}^{\alpha}\left(  x-d\right)  ^{\beta}\right]  (t)=J_{d^{+}%
}^{\alpha}(t)$. Hence,%
\begin{equation}
J_{d^{+}}^{\alpha}(t)=\Gamma\left(  \beta+1\right)  \epsilon^{\alpha+\beta
}\left(  z-1\right)  ^{\alpha}\sum_{k=0}^{\mathbf{\infty}}\frac{\left(
z-1\right)  ^{k}}{\Gamma\left(  \beta-k+1\right)  \Gamma\left(  \alpha
+k+1\right)  }. \label{A2}%
\end{equation}

Then using the Pochhammer symbols notation and the identity%
\[
\frac{\left(  -\beta\right)  _{k}}{\left(  \alpha+1\right)  _{k}}=\left(
-1\right)  ^{k}\frac{\Gamma\left(  \beta+1\right)  }{\Gamma\left(
\beta-k+1\right)  }\frac{\Gamma\left(  \alpha+1\right)  }{\Gamma\left(
\alpha+k+1\right)  },
\]
which implies%
\begin{equation}
\frac{1}{\Gamma\left(  \beta-k+1\right)  \Gamma\left(  \alpha+k+1\right)
}=\frac{\left(  -1\right)  ^{k}}{\Gamma\left(  \beta+1\right)  \Gamma\left(
\alpha+1\right)  }\frac{\left(  -\beta\right)  _{k}}{\left(  \alpha+1\right)
_{k}}, \label{A3}%
\end{equation}
we have, after substituting Eq.(\ref{A3}) in Eq.(\ref{A2}) that%
\[
J_{d^{+}}^{\alpha}(t)=\frac{\epsilon^{\alpha+\beta}\left(  z-1\right)
^{\alpha}}{\Gamma\left(  \alpha+1\right)  }\sum_{k=0}^{\mathbf{\infty}}%
\frac{\left(  -1\right)  ^{k}\left(  -\beta\right)  _{k}}{\left(
\alpha+1\right)  _{k}}\left(  z-1\right)  ^{k},
\]
and since $\left(  1\right)  _{k}=k!$ we have%
\begin{equation}
J_{d^{+}}^{\alpha}(t)=\frac{\epsilon^{\alpha+\beta}\left(  z-1\right)
^{\alpha}}{\Gamma\left(  \alpha+1\right)  }\sum_{k=0}^{\mathbf{\infty}}%
\frac{\left(  1\right)  _{k}\left(  -\beta\right)  _{k}}{\left(
\alpha+1\right)  _{k}}\frac{\left(  1-z\right)  ^{k}}{k!}. \label{A4}%
\end{equation}

Now, we recall that the hypergeometric function is defined by the series%
\[
_{2}F_{1}\left(  a,b;c;\xi\right)  =\sum_{k=0}^{\mathbf{\infty}}\frac{\left(
a\right)  _{k}\left(  b\right)  _{k}}{\left(  c\right)  _{k}}\frac{\xi^{k}%
}{k!},
\]
thus Eq.(\ref{A4}) can be written in the form%
\begin{equation}
J_{d^{+}}^{\alpha}(t)=\frac{\epsilon^{\alpha+\beta}\left(  z-1\right)
^{\alpha}}{\Gamma\left(  \alpha+1\right)  }\text{ }_{2}F_{1}\left(
1,-\beta;\alpha+1;1-z\right)  . \label{A5}%
\end{equation}

We can express this hypergeometric function (conveniently) as a series if, and
only if, $\alpha+\beta\neq\pm m$, $m\in%
\mathbb{N}
$ and $\left\vert \arg(z)\right\vert <\pi$ and we can rewrite this
hypergeometric function as \cite{Magnus}:%
\begin{align}
_{2}F_{1}\left(  1,-\beta;\alpha+1;1-z\right)   &  =\frac{\Gamma\left(
\alpha+1\right)  \Gamma\left(  \alpha+\beta\right)  }{\Gamma\left(
\alpha\right)  \Gamma\left(  \alpha+\beta+1\right)  }\text{ }_{2}F_{1}\left(
1,-\beta;1-\alpha-\beta;z\right)  +\label{A6}\\
&  \frac{\Gamma\left(  \alpha+1\right)  \Gamma\left(  -\alpha-\beta\right)
z^{\alpha+\beta}}{\Gamma\left(  -\beta\right)  }\text{ }_{2}F_{1}\left(
\alpha,\alpha+\beta+1;\alpha+\beta+1;z\right)  .\nonumber
\end{align}

Substituting Eq.(\ref{A6}) in Eq.(\ref{A5}) we have%
\begin{align}
J_{d^{+}}^{\alpha}(t)  &  =\frac{\epsilon^{\alpha+\beta}\left(  z-1\right)
^{\alpha}\Gamma\left(  \alpha+\beta\right)  }{\Gamma\left(  \alpha\right)
\Gamma\left(  \alpha+\beta+1\right)  }\text{ }_{2}F_{1}\left(  1,-\beta
;1-\alpha-\beta;z\right) \nonumber\\
&  +\frac{\epsilon^{\alpha+\beta}\left(  z-1\right)  ^{\alpha}\Gamma\left(
-\alpha-\beta\right)  z^{\alpha+\beta}}{\Gamma\left(  -\beta\right)  }\text{
}_{2}F_{1}\left(  \alpha,\alpha+\beta+1;\alpha+\beta+1;z\right)  . \label{A7}%
\end{align}

We can continue simplifying this last expression, using some identities for
the hypergeometric functions. First, for the second hypergeometric function in
Eq.(\ref{A7}), we have:%
\begin{align}
_{2}F_{1}\left(  \alpha,\alpha+\beta+1;\alpha+\beta+1;z\right)   &
=\sum_{k=0}^{\mathbf{\infty}}\left(  \alpha\right)  _{k}\frac{z^{k}}%
{k!}\nonumber\\
&  =\sum_{k=0}^{\mathbf{\infty}}\frac{\Gamma\left(  \alpha+k\right)  }%
{\Gamma\left(  \alpha\right)  k!}z^{k}\nonumber\\
&  =\sum_{k=0}^{\mathbf{\infty}}\left(  -1\right)  ^{k}\left(
\begin{array}
[c]{c}%
\alpha+k-1\\
k
\end{array}
\right)  \left(  -z\right)  ^{k}\nonumber\\
&  =\left(  1-z\right)  ^{-\alpha}\text{, }\left\vert z\right\vert <1.
\label{A8}%
\end{align}

While for the first hypergeometric function in Eq.(\ref{A7}) we use the
relation known as Euler transformation%
\[
_{2}F_{1}\left(  a,b;c;z\right)  =\left(  1-z\right)  ^{c-a-b}\text{ }%
_{2}F_{1}\left(  c-a;c-b,c;z\right)  ,
\]
and identify that%
\begin{align*}
c-a  &  =1-\alpha-\beta-1=-\alpha-\beta,\\
c-b  &  =1-\alpha-\beta+\beta=1-\alpha,\\
c-a-b  &  =1-\alpha-\beta-1+\beta=-\alpha.
\end{align*}

Therefore,
\[
_{2}F_{1}\left(  1,-\beta;1-\alpha-\beta;z\right)  =\left(  1-z\right)
^{-\alpha}\text{ }_{2}F_{1}\left(  -\alpha-\beta;1-\alpha,1-\alpha
-\beta;z\right)  ,
\]
and it follows that Eq.(\ref{A7}) assume the following form%
\begin{align}
J_{d^{+}}^{\alpha}(t)  &  =\frac{\left(  -1\right)  ^{\alpha}\Gamma\left(
-\alpha-\beta\right)  }{\Gamma\left(  -\beta\right)  }\left(  t-d\right)
^{\alpha+\beta}\nonumber\\
&  +\frac{\left(  -1\right)  ^{\alpha}\Gamma\left(  \alpha+\beta\right)
\epsilon^{\alpha+\beta}}{\Gamma\left(  \alpha\right)  \Gamma\left(
\alpha+\beta+1\right)  }\text{ }_{2}F_{1}\left(  -\alpha-\beta,1-\alpha
;1-\alpha-\beta;\frac{t-d}{\epsilon}\right)  . \label{A9}%
\end{align}

In a similar fashion for Eq.(\ref{Ia}), calling $\left[  _{a}J_{t}^{\alpha
}\left(  x-d\right)  ^{\beta}\right]  (t)=J_{a}^{\alpha}(t)$ we have%
\begin{align}
J_{a}^{\alpha}(t)  &  =\frac{\left(  -1\right)  ^{\alpha}\Gamma\left(
-\alpha-\beta\right)  }{\Gamma\left(  -\beta\right)  }\left(  t-d\right)
^{\alpha+\beta}\nonumber\\
&  +\frac{\left(  -1\right)  ^{\alpha}\Gamma\left(  \alpha+\beta\right)
\left(  a-d\right)  ^{\alpha+\beta}}{\Gamma\left(  \alpha\right)
\Gamma\left(  \alpha+\beta+1\right)  }\text{ }_{2}F_{1}\left(  -\alpha
-\beta,1-\alpha;1-\alpha-\beta;\frac{t-d}{a-d}\right)  . \label{A10}%
\end{align}

In a similar fashion for Eq.(\ref{Db}), calling $\left[  _{d^{+}}D_{t}%
^{\alpha}\left(  x-d\right)  ^{\beta}\right]  (t)=D_{d^{+}}^{\alpha}(t)$ we
have%
\begin{align}
D_{d^{+}}^{\alpha}(t)  &  =\frac{\left(  -1\right)  ^{-\alpha}\Gamma\left(
\alpha-\beta\right)  }{\Gamma\left(  -\beta\right)  }\left(  t-d\right)
^{-\alpha+\beta}\nonumber\\
&  +\frac{\left(  -1\right)  ^{-\alpha}\Gamma\left(  -\alpha+\beta\right)
\epsilon^{-\alpha+\beta}}{\Gamma\left(  -\alpha\right)  \Gamma\left(
-\alpha+\beta+1\right)  }\text{ }_{2}F_{1}\left(  \alpha-\beta,1+\alpha
;1+\alpha-\beta;\frac{t-d}{\epsilon}\right)  . \label{A11}%
\end{align}

In a similar fashion for Eq.(\ref{Da}), calling $\left[  _{a}D_{t}^{\alpha
}\left(  x-d\right)  ^{\beta}\right]  (t)=D_{a}^{\alpha}(t)$ we have%
\begin{align}
D_{a}^{\alpha}(t)  &  =\frac{\left(  -1\right)  ^{-\alpha}\Gamma\left(
\alpha-\beta\right)  }{\Gamma\left(  -\beta\right)  }\left(  t-d\right)
^{-\alpha+\beta}\nonumber\\
&  +\frac{\left(  -1\right)  ^{-\alpha}\Gamma\left(  -\alpha+\beta\right)
\epsilon^{-\alpha+\beta}}{\Gamma\left(  -\alpha\right)  \Gamma\left(
-\alpha+\beta+1\right)  }\text{ }_{2}F_{1}\left(  \alpha-\beta,1+\alpha
;1+\alpha-\beta;\frac{t-d}{a-d}\right)  . \label{A12}%
\end{align}

\end{document}